\documentclass[12pt]{amsart}

\usepackage{amssymb}
\usepackage{verbatim}
\usepackage[toc,page]{appendix}
\usepackage{mathrsfs}

\newtheorem{thm}{Theorem}[section]

\newtheorem{lem}[thm]{Lemma}
\newtheorem{cor}[thm]{Corollary}
\newtheorem{rem}[thm]{Remark}

\theoremstyle{definition}

\theoremstyle{remark}

\numberwithin{equation}{section}

\newcommand{\Z}{\mathbf{Z}}

\newcommand{\Mod}[1]{\ (\textup{mod}\ #1)}

\providecommand{\sgn}{\operatorname{sgn}}

\providecommand{\arcosh}{\operatorname{arcosh}}

\providecommand{\sym}{\operatorname{sym}}
\providecommand{\Aut}{\operatorname{Aut}}
\providecommand{\tr}{\operatorname{tr}}

\begin{document}

\title[]{Convolution formula for the sums of generalized Dirichlet $L$-functions}

\author{Olga  Balkanova}
\address{Department of Mathematical Sciences, University of Gothenburg and Chalmers University of Technology, SE-412 96 G\"{o}teborg, Sweden}
\email{olgabalkanova@gmail.com}

\author{Dmitry  Frolenkov}
\address{
Khabarovsk Division of the Institute for Applied Mathematics, Far Eastern Branch, Russian Academy of Sciences and Steklov Mathematical Institute of Russian Academy of
Sciences, 8 Gubkina st., Moscow, 119991, Russia}
\email{frolenkov@mi.ras.ru}
\thanks{Research of Dmitry Frolenkov is supported by the Russian Science Foundation under grant [14-11-00335] and performed in Khabarovsk Division of the Institute for Applied Mathematics, Far Eastern Branch, Russian Academy of Sciences}
\begin{abstract}
Using the Kuznetsov trace formula, we prove a spectral decomposition for the sums of generalized Dirichlet $L$-functions. Among applications are an explicit formula relating norms of prime geodesics to moments of symmetric square $L$-functions and an asymptotic expansion for the average of central values of generalized Dirichlet $L$-functions.
\end{abstract}


\maketitle

\tableofcontents


\section{Introduction}
The purpose of this paper is to advance understanding of the sums of the generalized Dirichlet $L$-functions
\begin{equation}\label{Lbyk}
\mathscr{L}_{n}(s)=\frac{\zeta(2s)}{\zeta(s)}\sum_{q=1}^{\infty}\frac{\rho_q(n)}{q^{s}},
\end{equation}
where $\zeta(s)$ is the Riemann zeta function and for $n \in \Z$
\begin{equation*}
\rho_q(n):=\#\{x\Mod{2q}:x^2\equiv n\Mod{4q}\}.
\end{equation*}
The function $\mathscr{L}_{n}(s)$, which is nonzero only if $n\equiv 0,1\pmod{4}$,  can be regarded as a generalization of both the Riemann zeta-function and the Dirichlet $L$-function $L(s,\chi_D)$ for primitive quadratic character $\chi_D$. If $n=0$ we have
\begin{equation*}
\mathscr{L}_{n}(s)=\zeta(2s-1).
\end{equation*}
If $D$ is a fundamental discriminant, then
\begin{equation*}
\mathscr{L}_{D}(s)=L(s,\chi_D).
\end{equation*}

The main result of the paper is the following convolution formula.
\begin{thm}\label{thm:convform} Let $\omega \in C^{\infty}$ be of compact support on $[a_1,a_2]$ for $0<a_1<a_2<\infty$. Assume that $\Re{s}>1$. Then
\begin{equation}\label{eq:spectrdecomp1}
\sum_{n=1}^{\infty}\omega(n)\mathscr{L}_{n^2-4l^2}(s)=MT(s)+Z_C(s)+Z_H(s)+ Z_D(s),
\end{equation}
where
\begin{equation}
MT(s)=
\sigma_{-s}(l^2)\frac{\zeta(2s)}{\zeta(1+s)}\int_0^{\infty}\omega(x)dx,
\end{equation}
\begin{equation}\label{eq:zc}
Z_C(s)=\frac{\zeta(s)}{\pi^{1/2+s}}\int_{-\infty}^{\infty}\frac{\sigma_{2ir}(l^2)}{l^{2ir}}
\frac{\zeta(s+2ir)\zeta(s-2ir)}{\left|\zeta(1+2ir\right)|^2}h(\omega;s;r)dr,
\end{equation}
\begin{equation}\label{eq:zh}
Z_H(s)=\pi^{1/2-s}\sum_{k\geq 6}g(\omega;s;k)\sum_{j\leq \vartheta(k)}\alpha_{j,k}t_{j,k}(l^2)L(\sym^2u_{j,k}, s),
\end{equation}
\begin{equation}\label{eq:zd}
Z_D(s)=\pi^{1/2-s}\sum_{j}\alpha_jt_j(l^2)L(\sym^2u_j, s)h(\omega;s;\kappa_j),
\end{equation}
where
\begin{equation}\label{eq:f2}
f(\omega;s;x)=\frac{2}{\pi^{1/2}}
\left(\frac{x}{4l}\right)^{s}
\int_{0}^{\infty}\omega(y)\cos\left(\frac{xy}{2l}\right)dy,
\end{equation}
\begin{equation}\label{eq:g}
g(\omega;s;k)=\pi(-1)^k\int_{0}^{\infty}J_{2k-1}(x)\frac{f(\omega;s;x)}{x}dx,
\end{equation}
\begin{equation}\label{eq:h}
h(\omega;s;r)=\pi\int_{0}^{\infty}k_0(x,ir)\frac{f(\omega;s;x)}{x}dx.
\end{equation}
\end{thm}

Furthermore, we prove the analytic continuation of \eqref{eq:spectrdecomp1} to regions containing the special points $s=1/2$ and $s=1$.
This is of particular importance since special values of $\mathscr{L}_{n}(s)$ arise naturally in various contexts, from the theory of modular forms to the Prime Geodesic Theorem.

First,  the function $\mathscr{L}_{n}(s)$ appears in the Fourier-Whittaker expansion of the combination of Maass-Eisenstein series of half-integral weight and level $4$ at cusps $0$ and $\infty$. See \cite{C,GH, PRR}.

Second, $\mathscr{L}_{n}(s)$ was studied by Zagier \cite{Z} in relation to the zeta function
\begin{equation*}
\zeta(s,\delta)=\sum_{\substack{\phi\pmod{\Gamma}\\ |\phi|=\delta}}\sum_{\substack{(m,n)\in \Z^2/\Aut(\phi)\\\phi(m,n)>0}}\frac{1}{\phi(m,n)^s},
\end{equation*}
where for  the full modular group $\Gamma$ the outer sum is over all $\Gamma$-equivalence classes of forms $\phi$ of discriminant $\delta$ and
the inner sum is over equivalence classes of pairs of integers modulo the group of automorphs of the form $\phi$.
In particular, Zagier constructed a modular form whose Fourier coefficients are infinite linear combinations of the zeta functions $\zeta(s,\delta)$ and proved that the Petersson inner product of the resulting modular form with an arbitrary primitive cusp form is the corresponding Rankin-Selberg zeta function. Furthermore, using the theory of binary quadratic forms, Zagier showed that
\begin{equation*}\zeta(s,\delta)=\zeta(s)\mathscr{L}_{\delta}(s).
\end{equation*}
 Finally, Zagier proved the analytic continuation and the functional equation
\begin{equation*}
\mathscr{L}_{n}^{*}(s)=\mathscr{L}_{n}^{*}(1-s)
\end{equation*}
of the completed $L$-function
\begin{equation*}
\mathscr{L}_{n}^{*}(s):=(\pi/|n|)^{-s/2}\Gamma(s/2+1/4-\sgn{n}/4)\mathscr{L}_{n}(s).
\end{equation*}

In the subsequent paper \cite{K},  Kuznetsov investigated the relation between $\mathscr{L}_{n}(s)$  and the Selberg trace formula. More precisely, the main result of \cite{K} is a new version of the trace formula, where a series involving $\mathscr{L}_{n^2-4}(1)$ replaces the sum over primitive hyperbolic classes of the modular group. Most importantly, Kuznetsov proved that
\begin{equation}\label{eq:Kuz}
\Psi_{\Gamma}(x)=2\sum_{n\leq X}\sqrt{n^2-4}\mathscr{L}_{n^2-4}(1), \quad X=x^{1/2}+x^{-1/2},
\end{equation}
where $\Psi_{\Gamma}(x)=\sum_{NP\leq x}\Lambda(P)$, $NP$ is the norm of the hyperbolic conjugacy class $\{P\}$ and $\Lambda(P)=\log{NP_0}$ if $\{P\}$ is a power of the primitive hyperbolic class $\{P_0\}$.

Equation  \eqref{eq:Kuz} turns out to be of crucial importance in the study of the Prime Geodesic Theorem.
Using   \eqref{eq:Kuz}, Bykovskii \cite{B} proved the Prime Geodesic Theorem in short intervals.
Moreover, the proof of Soundararajan and Young \cite{SY} of the best known (up to date) error term in the classical Prime Geodesic Theorem is based on \eqref{eq:Kuz}.

In addition, the special value  $\mathscr{L}_{n^2-4}(1)$ serves as a connecting link between prime geodesics and Kloosterman sums. Namely,
\cite[Theorem~1.3]{SY} states that for $\omega$ smooth, even, compactly supported function such that $\omega(x)=0$ if $|x|\leq 2$ we have
\begin{multline}\label{eq:SY}
\sum_{\{P\}}\Lambda(P)\frac{\omega(\tr(P))}{\sqrt{\tr(P)^2-4}}=\sum_{n=3}^{\infty}2\omega(n)\mathscr{L}_{n^2-4}(1)\\=\zeta(2)\sum_{q=1}^{\infty}\frac{1}{q^2}\sum_{l=-\infty}^{+\infty}S(l^2,1;q)\breve{\omega}\left( \frac{l}{q}\right),
\end{multline}
where the sum $\sum_{\{P\}}$ is over all hyperbolic conjugacy classes, $\tr(P)$ is the trace of conjugacy class $\{P\}$ and $\breve{\omega}(y)$ is the Fourier transform of $\omega(x)$. Note that the convergence of the sums over $q$ and $l$ in equation \eqref{eq:SY} is not obvious. For this reason, as indicated by the authors  of \cite{SY}, the formula \eqref{eq:SY} may not be useful for applications. This difficulty can be overcome by further investigation of the corresponding sums of Kloosterman sums, which is one of the results of the present paper.
\begin{thm}\label{thm:at1}  Let $\omega(x)$ be as in Theorem \ref{thm:convform} with $a_1>2$. Then
\begin{multline*}
\frac{1}{2}\sum_{\{P\}}\Lambda(P)\frac{\omega(\tr(P))}{\sqrt{\tr(P)^2-4}}=\sum_{n=3}^{\infty}\omega(n)\mathscr{L}_{n^2-4}(1)\\
=\widetilde{\omega}(1)+Z_C(1)+Z_H(1)+ Z_D(1),
\end{multline*}
where $Z_H(s)$, $Z_D(s)$ are defined by \eqref{eq:zh}, \eqref{eq:zd} and
\begin{multline*}
Z_C(1)=-\frac{1}{2\sqrt{\pi}}h(\omega;1;0)
+\frac{1}{\pi^{3/2}}\int_{-\infty}^{+\infty}\frac{\partial}{\partial s}h\left( \omega;s;r\right)\biggr|_{1}dr
\\+\frac{1}{\pi^{3/2}}\int_{-\infty}^{+\infty}\biggl(
\frac{\zeta'(1+2ir)}{\zeta(1+2ir)} +\frac{\zeta'(1-2ir)}{\zeta(1-2ir)}\biggr)h(\omega;1;r)dr.
\end{multline*}
\end{thm}

Finally, the  special value of \eqref{Lbyk} at the point $s=1$ appears in the elliptic part of the Arthur-Selberg trace formula, as has been shown by Altug in \cite[Eq.~4]{Al}.

Another special value $ \mathscr{L}_{n^2-4l^2}(1/2)$ is also related to some important problems in analytic number theory. There are several reasons for inquiring into investigation of the sums
\begin{equation}\label{eq:form}
\sum_{n=1}^{\infty}\omega(n)\mathscr{L}_{n^2-4l^2}(1/2), \quad l \geq 1
\end{equation}
for some specified function $\omega(n)$.

Firstly, the exact formula for the first moment of symmetric square $L$-functions attached to holomorphic cusp forms of level $1$ and large weight
involves sums of the form \eqref{eq:form}, as shown in \cite[Theorem~2.1]{BF1}.

Secondly, studying \eqref{eq:form}  is important for understanding the combinatorial structure of the second  moment of the symmetric square $L$-functions.

Thirdly, this study explores the idea that the spectral decompositions of the sums
\eqref{eq:form} and
\begin{equation}\label{eq:adddiv}
\sum_{n=1}^{\infty}\omega(n)\tau(n)\tau(n+l) , \quad l\geq 1, \quad \tau(n)=\sum_{d|n}1,
\end{equation}
are similar. Compare \cite[Theorem~3]{Mot} and Theorem \ref{thm:at12}.
In particular, assuming the Lindel\"{o}f hypothesis
\begin{equation*}
\mathscr{L}_{n}(1/2)\ll n^{\epsilon},\quad \epsilon>0,
\end{equation*}
we obtain an asymptotic formula for
\begin{equation*}
\sum_{n\leq X}\mathscr{L}_{n^2-4}(1/2)
\end{equation*}
with the same error term as in the corresponding analysis of the binary additive divisor problem.
\begin{thm}\label{lem:convsum}
For any $\epsilon>0$
\begin{multline}\label{eq:app1}
\sum_{2<n<X}\mathscr{L}_{n^2-4}(1/2)=\frac{X\log{X}}{\zeta(3/2)}+\frac{X}{2\zeta(3/2)}
\Biggl( -2-\frac{\pi}{2}+3\gamma-2\frac{\zeta'(3/2)}{\zeta(3/2)}-\log{8\pi}
\Biggr)\\+O(X^{2/3+2\theta/3+\epsilon}),
\end{multline}
where $\theta$ is the best known result towards the Lindel\"{o}f hypothesis for Dirichlet $L$-functions of real primitive characters.
The current record $\theta=1/6$ is due to Conrey and Iwaniec \cite{CI}.
\end{thm}

Lastly, the investigation of
\begin{equation}\label{eq:shortint}
\sum_{X<n\leq X+T}\mathscr{L}_{n^2-4}(1/2+ir), \quad |r|<X^{\epsilon}
\end{equation}
 is ultimately connected to the quality of the error term in the Prime Geodesic Theorem. In particular, it follows from
\cite{SY} that if \eqref{eq:shortint} can be bounded by $T$ for any $T\gg X^{2/3}$, then
\begin{equation}\label{eq:prg}
\Psi_{\Gamma}(x)=x+O(x^{2/3+\epsilon}).
\end{equation}
Unfortunately, the required estimate for \eqref{eq:shortint} is out of reach by our methods. Nevertheless, we prove a smoothed version of this result.
\begin{thm}\label{lem:exp}
For any $\epsilon>0$ and $T>X^{2/3+\epsilon}$ we have
\begin{equation}\label{eq:expshortint}
\sum_{n>2}\mathscr{L}_{n^2-4}(1/2)\exp\left( -\left(\frac{n-X}{T} \right)^2\right)\ll
T.
\end{equation}
\end{thm}
With slightly more careful calculations, it is possible to replace $\mathscr{L}_{n^2-4}(1/2)$ by $\mathscr{L}_{n^2-4}(1/2+ir)$, $|r|<X^{\epsilon}$ in \eqref{eq:expshortint}. However, it is unclear how to remove the exponential multiple.

\section{Notation}
Define the Mellin transform of  $f(x)$ as follows
\begin{equation*}
\tilde{f}(s)=\int_0^{\infty}f(x)x^{s-1}dx.
\end{equation*}
Let $\{u_{j,k} :1\leq j\leq \vartheta(k)\}$ be  the orthonormal basis of the space of holomorphic cusp forms of weight $2k$ and level $1$ consisting of eigenfunctions of all Hecke operators. Any element of this basis has the following Fourier expansion
\begin{equation*}
u_{j,k}(z)=\sum_{n=1}^{\infty}\rho_{j,k}(n)e(nz),\quad e(x):=\exp(2\pi i x),
\end{equation*}
where
\begin{equation*}
\rho_{j,k}(n)=\rho_{j,k}(1)t_{j,k}(n)n^{k-1/2}
\end{equation*}
and $\{t_{j,k}(n)\}$ are the eigenvalues of Hecke operators acting on $u_{j,k}(z)$.
Note that
\begin{equation}\label{eq:multipFourcoeff}
t_{j,k}(n)t_{j,k}(m)=\sum_{d|(m,n)}t_{j,k}\left( \frac{nm}{d^2}\right).
\end{equation}

Let us introduce the normalizing coefficient
\begin{equation*}
\alpha_{j,k}:=\frac{16\Gamma(2k)}{(4\pi)^{2k+1}}|\rho_{j,k}(1)|^2,
\end{equation*}
where $\Gamma(s)$ is the Gamma function.

For $\Re{s}>1$ define the symmetric square $L$-function
\begin{equation*}
L(\sym^2 u_{j,k},s):=\zeta(2s)\sum_{n=1}^{\infty}\frac{t_{j,k}(n^2)}{n^s}.
\end{equation*}

Let $\{u_j\}$ be the orthonormal basis of the space of Maass cusp forms consisting of common eigenfunctions of all Hecke operators and the hyperbolic Laplacian. Denote
$\{t_{j}(n)\}$ the eigenvalues of Hecke operators acting on $u_{j}$ and $\lambda_{j}=1/4+\kappa_{j}^2$ the eigenvalues of the hyperbolic Laplacian acting on $u_{j}$.

Any element of this basis has the following Fourier expansion
\begin{equation*}
u_{j}(x+iy)=\sqrt{y}\sum_{n\neq 0}\rho_{j}(n)K_{i\kappa_j}(2\pi|n|y)e(nx),
\end{equation*}
where $K_{\alpha}(x)$ is the $K$-Bessel function and
\begin{equation*}
\rho_{j}(n)=\rho_{j}(1)t_{j}(n).
\end{equation*}

Similarly to the holomorphic case, we have
\begin{equation}\label{eq:multipFourcoeff2}
t_{j}(n)t_{j}(m)=\sum_{d|(m,n)}t_{j}\left( \frac{nm}{d^2}\right).
\end{equation}

The normalizing coefficient is given by
\begin{equation*}
\alpha_{j}:=\frac{|\rho_{j}(1)|^2}{\cosh{\pi \kappa_j}}.
\end{equation*}

For $\Re{s}>1$ define the symmetric square $L$-function
\begin{equation*}
L(\sym^2 u_{j},s):=\zeta(2s)\sum_{n=1}^{\infty}\frac{t_{j}(n^2)}{n^s}.
\end{equation*}

It follows from \cite[Theorem~7.1.2]{NMH} and \cite[Theorem~1]{TX} that
\begin{equation}\label{eq:estsym2}
\sum_{V<|\kappa_j|<2V}\alpha_j|L(\sym^2u_j,1/2)|\ll V^{2+\epsilon}.
\end{equation}

Let $F(a,b,c;x)$ be the Gauss hypergeometric function and
\begin{equation}\label{eq: psisym2}
\Psi_k(x):=x^k\frac{\Gamma(k-1/4)\Gamma(k+1/4)}{\Gamma(2k)}F\left(k-\frac{1}{4},k+\frac{1}{4},2k;x \right).
\end{equation}

The sum of positive divisors function is defined as
$$\sigma_s(n):=\sum_{d|n}d^s.$$

The Lerch zeta function
\begin{equation*}
\zeta(\alpha,\beta;s):=\sum_{n+\alpha>0}\frac{e(n\beta)}{(n+\alpha)^s},\quad \Re{s}>1,
\end{equation*}
satisfies the functional equation (see \cite{Ler})
\begin{multline}\label{eq:felerch}
\zeta(\beta,0;s)=(2\pi)^{s-1}\Gamma(1-s)\biggl[ -ie(s/4)\zeta(0,\beta;1-s)\\
+e(-s/4)\zeta(0,-\beta;1-s)\biggr].
\end{multline}

Finally, define the Bessel kernel
\begin{equation}\label{k0def}
k_0(x,v):=\frac{1}{2\cos{\pi (1/2+v)}}\left(J_{2v}(x)-J_{-2v}(x) \right),
\end{equation}
where $J_{v}(x)$ is the $J$-Bessel function.

\section{Generalized Kloosterman sums}\label{sec:klsums}

Define the generalized Kloosterman sum by
\begin{equation*}
S(m,n_1,n_2;q):=\sum_{a,b=1}^{q}\delta_q(ab-m)e\left( \frac{an_1+bn_2}{q}\right),
\end{equation*}
where $e(x):=\exp(2\pi i x)$ and
\begin{equation*}
\delta_q(ab-m)=\begin{cases}
1 & \text{ if } ab\equiv m\pmod{q}\\
0 & \text{ otherwise}
\end{cases} .
\end{equation*}
Note that
\begin{equation*}S(1,n_1,n_2;q)=S(n_1,n_2;q),
\end{equation*} where $S(n_1,n_2;q)$ is the classical Kloosterman sum.

\begin{lem}(Selberg's identity) We have
\begin{equation}\label{eq:BKV}
S(m,n_1,n_2;q)=\sum_{d|(m,n_1,q)}dS\left(1,\frac{mn_1}{d^2},n_2;\frac{q}{d} \right).
\end{equation}
\end{lem}
\begin{proof}
See \cite[Eq.~2.29, p.~48]{BKV} and \cite[p.~100]{A}.
\end{proof}

\begin{lem}\label{lem:Klperm}
The sum $S(m,n_1,n_2;q)$ is invariant under any permutation of $m$, $n_1$, $n_2$.
\end{lem}
\begin{proof}
The required property can be derived from the following identity
\begin{equation*}
S(m,n_1,n_2;q)=\frac{1}{q}\sum_{a,b,c=1}^{q}e\left( \frac{an_1+bn_2+cm-abc}{q}\right).
\end{equation*}
\end{proof}

\begin{lem}\label{Kuznetsov}(Kuznetsov trace formula)
Let $\phi \in C^3(0,\infty)$ such that for $\nu=0,1,2,3$ and arbitrary small constant $\epsilon>0$ we have
\begin{equation*}
\phi^{(\nu)}(x)\ll x^{1/2-\nu+\epsilon} \text{ as } x\rightarrow +0,
\end{equation*}
\begin{equation*}
\phi^{(\nu)}(x)\ll x^{-1-\nu-\epsilon} \text{ as } x\rightarrow +\infty.
\end{equation*}
Then for all $m,n \geq 1$
\begin{multline}\label{eq:KuzTrForm}
\sum_{q=1}^{\infty}\frac{S(m,n_1,n_2;q)}{q}\phi\left( \frac{4\pi \sqrt{mn_1n_2}}{q}\right)=
\sum_{j=1}^{\infty}\alpha_jt_j(m)t_j(n_1)t_j(n_2)\widehat{\phi}(\kappa_j)\\
+\sum_{k\geq 6}\sum_{j\leq \vartheta(k)}\alpha_{j,k}t_{j,k}(m)t_{j,k}(n_1)t_{j,k}(n_2)\widehat{\phi}(i(1/2-k))\\+
\frac{1}{\pi}\int_{-\infty}^{+\infty}\frac{\sigma_{2ir}(m)\sigma_{2ir}(n_1)\sigma_{2ir}(n_2)}{(mn_1n_2)^{ir}|\zeta(1+2ir)|^2}\widehat{\phi}(r)dr,
\end{multline}
where
\begin{equation*}
\widehat{\phi}(r)=\pi \int_{0}^{\infty}k_0(x,ir)\phi(x)\frac{dx}{x},
\end{equation*}
\begin{equation*}
\widehat{\phi}(i(1/2-k))=\pi (-1)^k \int_{0}^{\infty}J_{2k-1}(x)\phi(x)\frac{dx}{x}.
\end{equation*}
\end{lem}

\begin{proof}

Consider
\begin{multline*}
\sum_{q=1}^{\infty}\frac{S(m,n_1,n_2;q)}{q}\phi\left( \frac{4\pi \sqrt{mn_1n_2}}{q}\right)\\=\sum_{q=1}^{\infty}\frac{1}{q}\phi\left( \frac{4\pi \sqrt{mn_1n_2}}{q}\right)
\sum_{d|(m,n_1,q)}dS(1,mn_1/d^2,n_2;q/d)\\=
\sum_{d|(m,n_1)}\sum_{q=1}^{\infty}\frac{1}{q}S(1,mn_1/d^2,n_2;q)\phi\left( \frac{4\pi \sqrt{mn_1n_2}}{qd}\right).
\end{multline*}

Applying \cite[Lemma~1]{Mot} to the inner sum and using the identities \eqref{eq:multipFourcoeff} and \eqref{eq:multipFourcoeff2}, we prove the lemma.

\end{proof}

\section{Convolution formula}

In this section we prove Theorem \ref{thm:convform}. The methods we apply are similar to the ones  developed by Kuznetsov in  \cite{K2} while studying the double zeta function
\begin{equation}\label{eq:LKuzminus}
\sum_{n=1}^{\infty}\frac{L^{(-)}_n(s)}{n^{\alpha}},\quad L^{(-)}_n(s):=\sum_{c=1}^{\infty}\frac{a_n(c)}{c^s}, \quad \Re{\alpha},\Re{s}>1,
\end{equation}
where
$a_n(c)$ is the number of integral solutions of $$x^2+nx-1\equiv 0 \pmod{c}.$$

The key idea  is to obtain a representation of the generalized Dirichlet $L$-functions in terms of sums of Kloosterman sums, and then apply the Kuznetsov trace formula.

To guarantee the absolute convergence, we first assume that $\Re{s}>3/2$.

In order to evaluate $\sum_{n=1}^{\infty}\omega(n)\mathscr{L}_{n^2-4l^2}(s),$ we apply the Mellin transform for $\omega(n)$. Consequently, it is required to study the series of the form
$\sum_{n=1}^{\infty}\mathscr{L}_{n^2-4l^2}(s)n^{-\alpha}.$

\begin{lem}\label{lem:zetafunctL} For $\Re{s}>3/2$ and $\Re{\alpha}>1$
\begin{equation*}
\sum_{n=1}^{\infty}\frac{\mathscr{L}_{n^2-4l^2}(s)}{n^{\alpha}}=
\frac{\zeta(2s)}{\zeta(s)}\sum_{q=1}^{\infty}\frac{1}{q^{\alpha+s}}\sum_{a,b=1}^{q}\delta_q(ab-l^2)\zeta\left(-\frac{a+b}{q},0,\alpha\right).
\end{equation*}
\end{lem}
\begin{proof}
Note that there is a one-to-one correspondence between the solutions $x\Mod{2q}$ of $x^2\equiv n^2-4l^2 \Mod{4q}$ and the solutions $y\Mod{q}$ of $y^2+ny+l^2\equiv 0 \Mod{q}$. Therefore,
\begin{equation}\label{rho 1}
\rho_q(n^2-4l^2)=\sum_{\substack{y\Mod{q}\\ y^2+l^2+ny\equiv 0\Mod{q}}}1=
\sum_{\substack{x,y\Mod{q}\\ x+y\equiv -n\Mod{q}}}\delta_q(xy-l^2).
\end{equation}
Consequently,
\begin{multline}\label{rho 2}
\rho_q(n^2-4l^2)=\frac{1}{q}
\sum_{x,y\Mod{q}}\delta_q(xy-l^2)\sum_{c\Mod{q}}
e\left(\frac{cx+cy+cn}{q}\right)=\\
\frac{1}{q}\sum_{c\Mod{q}}
e\left(\frac{cn}{q}\right)S(c,c,l^2;q).
\end{multline}
According to \eqref{Lbyk}
\begin{equation*}
\sum_{n=1}^{\infty}\frac{\mathscr{L}_{n^2-4l^2}(s)}{n^{\alpha}}=
\frac{\zeta(2s)}{\zeta(s)}\sum_{n=1}^{\infty}\frac{1}{n^{\alpha}}\sum_{q=1}^{\infty}\frac{\rho_q(n^2-4l^2)}{q^{s}}.
\end{equation*}
Applying \eqref{rho 2} and Weil's bound for Kloosterman sums it follows that the sums above are absolutely convergent. Changing the order of summation and using \eqref{rho 1}, we obtain
\begin{multline*}
\sum_{n=1}^{\infty}\frac{\mathscr{L}_{n^2-4l^2}(s)}{n^{\alpha}}=
\frac{\zeta(2s)}{\zeta(s)}\sum_{q=1}^{\infty}\frac{1}{q^{s}}
\sum_{x,y\Mod{q}}\delta_q(xy-l^2)
\sum_{\substack{n=1\\ n\equiv -x-y\Mod{q}}}^{\infty}\frac{1}{n^{\alpha}}=\\
\frac{\zeta(2s)}{\zeta(s)}\sum_{q=1}^{\infty}\frac{1}{q^{s}}
\sum_{x,y\Mod{q}}\delta_q(xy-l^2)
\frac{1}{q^{\alpha}}\zeta\left(-\frac{x+y}{q},0,\alpha\right).
\end{multline*}
\end{proof}

In order to simplify the expression proved in the previous lemma, we need the following result.
\begin{lem}\label{lem:mterm} For $\Re{s}>1$ we have
\begin{equation*}
\sum_{q=1}^{\infty}\frac{1}{q^{1+s}}\sum_{a,b=1}^{q}\delta_q(ab-l^2)=\sigma_{-s}(l^2)\frac{\zeta(s)}{\zeta(1+s)}.
\end{equation*}
\end{lem}

\begin{proof}
Note that
\begin{equation*}
\sum_{a,b=1}^{q}\delta_q(ab-l^2)=S(l^2,0,0;q)=S(0,l^2,0;q)=\sum_{a,b=1}^{q}\delta_q(ab)e\left( \frac{al^2}{q}\right).
\end{equation*}
Thus
\begin{multline*}
\sum_{q=1}^{\infty}\frac{1}{q^{1+s}}\sum_{a,b=1}^{q}\delta_q(ab-l^2)=\sum_{q=1}^{\infty}\frac{1}{q^{1+s}}\sum_{\substack{a,b=1\\ab \equiv 0\pmod{q}}}^{q}e\left( \frac{al^2}{q}\right)=\\
\sum_{q=1}^{\infty}\frac{1}{q^{1+s}}\sum_{a=1}^{q}
\sum_{\substack{b=1\\b\equiv 0\pmod{q/(a,q)}}}^qe\left( \frac{al^2}{q}\right)=
\sum_{q=1}^{\infty}\frac{1}{q^{1+s}}\sum_{a=1}^{q}e\left( \frac{al^2}{q}\right)(a,q).
\end{multline*}
The inner sum can be expressed in terms of Ramanujan sums as follows
\begin{equation*}
\sum_{a=1}^{q}e\left( \frac{al^2}{q}\right)(a,q)=\sum_{d|q}d\sum_{\substack{a\pmod{q}\\ (a,q)=d}}e\left( \frac{al^2}{q}\right)
=\sum_{d|q}d S(0,l^2;q/d).
\end{equation*}
Consequently,
\begin{multline*}
\sum_{q=1}^{\infty}\frac{1}{q^{1+s}}\sum_{a,b=1}^{q}\delta_q(ab-l^2)=\sum_{q=1}^{\infty}\frac{1}{q^{1+s}}\sum_{d|q}d S(0,l^2;q/d)\\=\sum_{d=1}^{\infty}\frac{1}{d^s}\sum_{q=1}^{\infty}\frac{ S(0,l^2;q)}{q^{1+s}}=
\sigma_{-s}(l^2)\frac{\zeta(s)}{\zeta(1+s)}.
\end{multline*}
\end{proof}

Combining all the results, we express the sums of the generalized Dirichlet $L$-functions in terms of sums of Kloosterman sums.
\begin{lem}
For $\Re{s}>3/2$ we have
\begin{multline}\label{eq:finalexp}
\sum_{n=1}^{\infty}\omega(n)\mathscr{L}_{n^2-4l^2}(s)=\widetilde{\omega}(1)\sigma_{-s}(l^2)\frac{\zeta(2s)}{\zeta(1+s)} +\\
\frac{\zeta(2s)}{\zeta(s)}\pi^{1/2-s}\sum_{n=1}^{\infty}\frac{1}{n^s}
\sum_{q=1}^{\infty}\frac{1}{q}S(l^2,n,n;q)f\left(\omega;s;\frac{4\pi nl}{q}\right),
\end{multline}
where
\begin{equation}\label{eq:f}
f(\omega;s;x)=\frac{1}{2\pi i}\int_{(\delta)}\frac{\Gamma(1/2-\alpha/2)}{\Gamma(\alpha/2)}\widetilde{\omega}(\alpha)\left( \frac{x}{4l}\right)^{\alpha+s-1}d\alpha,
\end{equation}
where $\delta<1.$
\end{lem}

\begin{proof}

Applying Lemma \ref{lem:zetafunctL}  we obtain for $\sigma>1$ the following expression
\begin{multline*}
\sum_{n=1}^{\infty}\omega(n)\mathscr{L}_{n^2-4l^2}(s)=\\ \frac{\zeta(2s)}{\zeta(s)}
\frac{1}{2\pi i}\int_{(\sigma)}\widetilde{\omega}(\alpha)\sum_{q=1}^{\infty}\frac{1}{q^{\alpha+s}}\sum_{a,b=1}^{q}\delta_q(ab-l^2)\zeta\left( -\frac{a+b}{q},0;\alpha\right)d\alpha.
\end{multline*}
Moving the contour of integration to $-\epsilon$ for an arbitrary small $\epsilon>0$, we cross a pole at $\alpha=1$. Therefore,
\begin{multline}\label{sum L1}
\sum_{n=1}^{\infty}\omega(n)\mathscr{L}_{n^2-4l^2}(s)=\widetilde{\omega}(1)\frac{\zeta(2s)}{\zeta(s)}\sum_{q=1}^{\infty}\frac{1}{q^{1+s}}\sum_{a,b=1}^{q}\delta_q(ab-l^2)\\+\frac{\zeta(2s)}{\zeta(s)}
\frac{1}{2\pi i}\int_{(-\epsilon)}\widetilde{\omega}(\alpha)\sum_{q=1}^{\infty}\frac{1}{q^{\alpha+s}}\sum_{a,b=1}^{q}\delta_q(ab-l^2)\zeta\left( -\frac{a+b}{q},0;\alpha\right)d\alpha.
\end{multline}

By Lemma \ref{lem:mterm} we have
\begin{equation}\label{sum Lmainterm}
\widetilde{\omega}(1)\frac{\zeta(2s)}{\zeta(s)}\sum_{q=1}^{\infty}\frac{1}{q^{1+s}}\sum_{a,b=1}^{q}\delta_q(ab-l^2)=
\widetilde{\omega}(1)\sigma_{-s}(l^2)\frac{\zeta(2s)}{\zeta(1+s)}.
\end{equation}

Applying the functional equation \eqref{eq:felerch} for the Lerch zeta function, we obtain
\begin{multline*}
\sum_{a,b=1}^{q}\delta_q(ab-l^2)\zeta\left( -\frac{a+b}{q},0;\alpha\right)=(2\pi)^{\alpha-1}\Gamma(1-\alpha)\sum_{n=1}^{\infty}\frac{1}{n^{1-\alpha}}\times \\
\biggl( -ie(\alpha/4)\sum_{a,b=1}^{q}\delta_q(ab-l^2)e\left( -\frac{na+nb}{q}\right)\\+ie(-\alpha/4)\sum_{a,b=1}^{q}\delta_q(ab-l^2)e\left( \frac{na+nb}{q}\right)\biggr).
\end{multline*}

Note that
\begin{equation*}
-ie(\alpha/4)+ie(-\alpha/4)=2\sin\frac{\pi\alpha}{2}.
\end{equation*}

Therefore,
\begin{multline*}
\sum_{a,b=1}^{q}\delta_q(ab-l^2)\zeta\left( -\frac{a+b}{q},0;\alpha\right)\\=
2(2\pi)^{\alpha-1}\sin\frac{\pi\alpha}{2}\Gamma(1-\alpha)\sum_{n=1}^{\infty}\frac{S(l^2,n,n;q)}{n^{1-\alpha}}.
\end{multline*}

Then \cite[Eqs.~5.5.3,~5.5.5]{HMF} yield
\begin{multline}\label{sum Lerrorterm}
\sum_{a,b=1}^{q}\delta_q(ab-l^2)\zeta\left( -\frac{a+b}{q},0;\alpha\right)\\=
\pi^{\alpha-1/2}\frac{\Gamma(1/2-\alpha/2)}{\Gamma(\alpha/2)}
\sum_{n=1}^{\infty}\frac{S(l^2,n,n;q)}{n^{1-\alpha}}.
\end{multline}

Substituting \eqref{sum Lmainterm} and \eqref{sum Lerrorterm} to \eqref{sum L1}, we obtain

\begin{multline*}
\sum_{n=1}^{\infty}\omega(n)\mathscr{L}_{n^2-4l^2}(s)=\widetilde{\omega}(1)\sigma_{-s}(l^2)\frac{\zeta(2s)}{\zeta(1+s)} +
\frac{\zeta(2s)}{\zeta(s)}\sum_{q=1}^{\infty}\frac{1}{q^s}\\ \times\sum_{n=1}^{\infty}\frac{S(l^2,n,n;q)}{n}
\frac{1}{2\pi i}\int_{(-\epsilon)} \pi^{\alpha-1/2}\frac{\Gamma(1/2-\alpha/2)}{\Gamma(\alpha/2)}\left( \frac{n}{q}\right)^{\alpha}\widetilde{\omega}(\alpha)d\alpha.
\end{multline*}
Applying equation \eqref{eq:f}, the assertion follows.
\end{proof}
According to \cite[Eqs. 5.5.3, 5.5.5, 5.9.6]{HMF}  the following identity holds for $0<\Re{s}<1$
\begin{equation}\label{cosMellin}
\int_0^{\infty}\cos(x)x^{s-1}dx=\Gamma(s)\cos(\pi s/2)=\sqrt{\pi}2^{s-1}\frac{\Gamma(s/2)}{\Gamma(1/2-s/2)}.
\end{equation}
 Applying \eqref{eq:f}, \eqref{cosMellin} and the Mellin inversion formula \cite[Eq. 1.14.36]{HMF}, we obtain \eqref{eq:f2}.

In order to prove Theorem \ref{thm:convform} we evaluate the sums of Kloosterman sums in \eqref{eq:finalexp} by using the Kuznetsov trace formula given by equation \eqref{eq:KuzTrForm}.
It follows from \eqref{eq:multipFourcoeff} that
\begin{multline*}
\sum_{n=1}^{\infty}\frac{t_{j,k}^{2}(n)}{n^s}=\sum_{n=1}^{\infty}\frac{1}{n^s}\sum_{d|n}t_{j,k}\left( \frac{n^2}{d^2}\right)=
\sum_{d=1}^{\infty}\frac{1}{d^s}\sum_{n=1}^{\infty}\frac{t_{j,k}( n^2)}{n^s}\\=\frac{\zeta(s)}{\zeta(2s)}L(\sym^2 u_{j,k},s).
\end{multline*}
Analogously,
\begin{equation*}
\sum_{n=1}^{\infty}\frac{t_{j}^{2}(n)}{n^s}=\frac{\zeta(s)}{\zeta(2s)}L(\sym^2 u_{j},s).
\end{equation*}
Finally,
\begin{equation*}
\sum_{n=1}^{\infty}\frac{\sigma_{2ir}^{2}(n)}{n^{s+2ir}}=\frac{\zeta(s+2ir)\zeta(s-2ir)\zeta^2(s)}{\zeta(2s)}.
\end{equation*}
This completes the proof of Theorem \ref{thm:convform} for $\Re{s}>3/2$. By analytic continuation
 the convolution formula is valid for $\Re{s}>1$.

\section{Some properties of the weight functions}

In this section we study the weight functions appearing in Theorem \ref{thm:convform}, namely  $g(\omega;s;k)$ and $h(\omega;s;k)$. Properties of these functions are important for the analytic continuation of the convolution formula  in Theorem \ref{thm:convform} to regions containing the special points $s=1$ and $s=1/2$.

\begin{lem}\label{lem:g} For $1-2k<\Re{s}<3/2$ We have
\begin{equation*}
g(\omega;s;k)=g_1(\omega;s;k)+g_2(\omega;s;k),
\end{equation*}
where
\begin{multline}\label{g1def}
g_1(\omega;s;k)=(2l)^{2k-1}\sin{\frac{\pi s}{2}}\int_{2l}^{\infty}\omega(x)\Gamma(k+s/2-1/2)\\
\times\frac{\Gamma(k+s/2)}{\Gamma(2k)} x^{1-2k-s}F(k+s/2-1/2,k+s/2,2k;4l^2/x^2)dx.
\end{multline}
\begin{multline}\label{g2def}
g_2(\omega;s;k)=
\frac{1}{(2l)^s}\cos{\frac{\pi s}{2}}\int_{0}^{2l}\omega(x)\Gamma(k+s/2-1/2)\\
\times\frac{\Gamma(1/2+s/2-k)}{\Gamma(1/2)} F(k+s/2-1/2,1/2+s/2-k,1/2;x^2/(4l^2))dx.
\end{multline}
\end{lem}
\begin{proof}
According to \eqref{eq:f2} and \eqref{eq:g} the following holds
\begin{multline*}
g(\omega;s;k)=\frac{2\pi^{1/2}(-1)^k}{(4l)^s}
\int_{0}^{\infty}\omega(y)
\int_{0}^{\infty}J_{2k-1}(x)x^{s-1}
\cos\left(\frac{xy}{2l}\right)dxdy\\
=g_1(\omega;s;k)+g_2(\omega;s;k),
\end{multline*}
where $g_1(\omega;s;k)$ corresponds to the integral over $y\geq 2l$ and $g_2(\omega;s;k)$ corresponds to the integral over $y> 2l.$

First, consider $g_2(\omega;s;k)$. Applying \cite[Eq.~6.699(2)]{GR}, we have for $1-2k<\Re{s}<3/2$
\begin{multline*}
g_2(\omega;s;k)=\frac{\pi^{1/2}(-1)^k}{(2l)^s}
\int_{0}^{2l}\omega(y)\frac{\Gamma(k+s/2-1/2)}{\Gamma(k+1/2-s/2)}\\\times
 F(k+s/2-1/2,1/2+s/2-k,1/2;y^2/(4l^2))dy.
\end{multline*}
Applying \cite[Eqs. 5.5.3, 5.4.6]{HMF}, we obtain \eqref{g2def}.

Second, consider $g_1(\omega;s;k)$. Using \cite[Eq.~6.699(2)]{GR}, we have for $1-2k<\Re{s}<3/2$
\begin{multline*}
g_1(\omega;s;k)=
\frac{\pi^{1/2}2^{2-2k}(2l)^{2k+s-1}}{(4l)^s}\sin{\frac{\pi s}{2}}
\int_{2l}^{\infty}\omega(y)\frac{\Gamma(2k+s-1)}{\Gamma(2k)} y^{1-2k-s}\\\times
F(k+s/2-1/2,k+s/2,2k;4l^2/y^2)dy.
\end{multline*}
Applying \cite[Eq. 5.5.5]{HMF} to $\Gamma(2k+s-1)$, we obtain \eqref{g1def}.
\end{proof}

\begin{lem}\label{lem:hcont}
For $0<\Re{s}<3/2$ we have
\begin{equation}\label{eq:hh1h2}
h(\omega;s;r)=h_1(\omega;s;r)+h_1(\omega;s;-r)+h_2(\omega;s;r),
\end{equation}
where
\begin{multline}\label{eq:hwsr}
h_1(\omega;s;r)=\int_{2l}^{\infty}\frac{\omega(x)}{2x^s}\frac{\sin{\pi(s/2-1/2+ir)}}{\sin{(\pi i r)}}\Gamma(1/2+s/2+ir)\\ \times \frac{\Gamma(s/2+ir)}{\Gamma(1+2ir)}\left( \frac{x}{2l}\right)^{-2ir}
F(s/2+ir,s/2+1/2+ir,1+2ir;4l^2/x^2)dx,
\end{multline}
\begin{multline}\label{h2def}
h_2(\omega;s;r)=\int_{0}^{2l}\frac{\omega(x)}{(2l)^s}\cos{\frac{\pi s}{2}} \frac{\Gamma(s/2+ir)\Gamma(s/2-ir)}{\Gamma(1/2)}\\ \times
F(s/2+ir,s/2-ir,1/2;x^2/(4l^2))dx.
\end{multline}
\end{lem}
\begin{proof}

According to \eqref{eq:f2} and \eqref{eq:h} we have
\begin{equation*}
h(\omega;s;r)=\frac{2\pi^{1/2}}{(4l)^s}
\int_{0}^{\infty}\omega(y)
\int_{0}^{\infty}k_0(x,ir)x^{s-1}
\cos\left(\frac{xy}{2l}\right)dxdy,
\end{equation*}
where $k_0(x,v)$ is the $K$-Bessel kernel.
Using \eqref{k0def}, we obtain \eqref{eq:hh1h2} with
\begin{equation*}
h_1(\omega;s;r)=\frac{2\pi^{1/2}}{2(4l)^s\cos(\pi/2+\pi ir)}
\int_{2l}^{\infty}\omega(y)
\int_{0}^{\infty}J_{2ir}(x)x^{s-1}
\cos\left(\frac{xy}{2l}\right)dxdy,
\end{equation*}

\begin{multline*}
h_2(\omega;s;r)=\frac{2\pi^{1/2}}{2(4l)^s\cos(\pi/2+\pi ir)}
\int_{0}^{2l}\omega(y)
\int_{0}^{\infty}\left(J_{2ir}(x)-J_{-2ir}(x)\right)\\\times
x^{s-1}\cos\left(\frac{xy}{2l}\right)dxdy.
\end{multline*}

Let us consider the function $h_2(\omega;s;r).$  Applying \cite[Eq.~6.699(2)]{GR}, we obtain for $0<\Re{s}<3/2$
\begin{multline*}
h_2(\omega;s;r)=\frac{-2^{s-1}\pi^{1/2}}{(4l)^s\sin(\pi ir)}
\int_{0}^{2l}\omega(y)F\left(\frac{s}{2}+ir,\frac{s}{2}-ir,\frac{1}{2};\frac{y^2}{4l^2}\right)
\\\times
\left(\frac{\Gamma(s/2+ir)}{\Gamma(1-s/2+ir)}-\frac{\Gamma(s/2-ir)}{\Gamma(1-s/2-ir)}\right)
dy.
\end{multline*}
Using \cite[Eq. 5.5.3]{HMF} for $\Gamma(1-s/2\pm ir)$ and \cite[Eq. 5.4.6]{HMF}, we prove \eqref{h2def}.

Let us consider the function $h_1(\omega;s;r).$ Using \cite[Eq.~6.699(2)]{GR}, we have for $0<\Re{s}<3/2$
\begin{multline*}
h_1(\omega;s;r)=\frac{\pi^{1/2}2^{-2ir}}{(4l)^s}
\int_{2l}^{\infty}\omega(y)F\left(\frac{s}{2}+ir,\frac{1+s}{2}+ir,1+2ir;\frac{4l^2}{y^2}\right)
\\\times
\frac{\sin{\pi(s/2-1/2+ir)}}{\sin{(\pi i r)}}\left(\frac{y}{2l}\right)^{-s-2ir}
\frac{\Gamma(s+2ir)}{\Gamma(1+2ir)}dy.
\end{multline*}
Applying \cite[Eq. 5.5.5]{HMF} for $\Gamma(s+2ir)$,  we obtain \eqref{eq:hwsr}.

\end{proof}
\begin{rem}
It is possible to obtain other representations for the test functions $h(\omega;s;r)$ and $g(\omega;s;k).$ For example, instead of using Lemma \ref{Kuznetsov} one can apply a version of the Kuznetsov formula proved by Chamizo and Raboso in  \cite{CR}. This approach might yield simpler expressions that are based solely on the Fourier type integrals.
\end{rem}

\begin{cor}\label{cor:h11}
Assume that $l=1$. Then
\begin{equation}\label{eq:hatone}
 h(\omega;1;r)=\sqrt{\pi}\int_{0}^{\infty}\omega(2\cosh{\xi})\cos{(2r\xi)}d\xi,
\end{equation}
\begin{equation}\label{eq:gatone}
g(\omega;1;k)=\pi^{1/2}\int_{2}^{\infty}\frac{\omega(x)}{\sqrt{x^2-4}}\frac{2^{2k-1}}{(x+\sqrt{x^2-4})^{2k-1}}dx.
\end{equation}
\end{cor}

\begin{proof}
Specializing equation \eqref{eq:hwsr} to $s=l=1$ and using \cite[Eqs.~5.5.5,~15.4.18]{HMF}, namely
\begin{equation*}
\Gamma(1+2ir)=\pi^{-1/2}2^{2ir}\Gamma(1/2+ir)\Gamma(1+ir),
\end{equation*}
\begin{equation*}
F\left( \frac{1}{2}+ir,1+ir,1+2ir;\frac{4}{x^2}\right)=\left(1-\frac{4}{x^2}\right)^{-1/2}
\left(\frac{1}{2}+\frac{1}{2} \left(1-\frac{4}{x^2}\right)^{1/2}\right)^{-2ir},
\end{equation*}
 we obtain
\begin{equation*}
h(\omega;1;r)=\pi^{1/2}\int_{2}^{\infty}\frac{\omega(x)}{\sqrt{x^2-4}}\cos{\left(2r\log{\left(\frac{x}{2}+
\frac{\sqrt{x^2-4}}{2} \right)}\right)}dx.
\end{equation*}
It follows from \cite[Eq.~4.37.19]{HMF} that
\begin{equation*}
\log{\left(\frac{x}{2}+
\frac{\sqrt{x^2-4}}{2} \right)}=\arcosh{\frac{x}{2}}.
\end{equation*}
Substituting this into the integral and making the change of variables $$\xi:=\arcosh{\frac{x}{2}},$$
we obtain \eqref{eq:hatone}. The proof of \eqref{eq:gatone} is similar.
\end{proof}

\begin{cor}\label{cor:h2i}
For $\Re{s}\le 1$ we have
\begin{multline*}
h\left(\omega;s;\frac{s-1}{2i}\right)=\frac{\Gamma(s-1/2)}{(2l)^{1-s}}\sin{\frac{\pi s}{2}}\int_{2l}^{\infty}
\omega(x)(x^2-4l^2)^{1/2-s}dx\\+\frac{\Gamma(s-1/2)}{(2l)^{1-s}}\cos{\frac{\pi s}{2}}\int_{0}^{2l}
\omega(x)(4l^2-x^2)^{1/2-s}dx.
\end{multline*}
\end{cor}
\begin{proof}
Specializing equation \eqref{eq:hh1h2}  to $r=(s-1)/(2i)$ gives
\begin{multline*}
h\left(\omega;s;\frac{s-1}{2i}\right)=\\ \int_{2l}^{\infty}
\frac{\omega(x)}{x^s}\cos{\frac{\pi(s-1)}{2}}\Gamma(s-1/2)\left(\frac{x}{2l} \right)^{1-s}
F\left(s-\frac{1}{2},s,s;\frac{4l^2}{x^2} \right)dx+\\ \int_{0}^{2l}
\frac{\omega(x)}{(2l)^s}\cos{\frac{\pi s}{2}}\Gamma(s-1/2)
F\left(s-\frac{1}{2},1/2,1/2;\frac{x^2}{4l^2} \right)dx.
\end{multline*}
Applying \cite[Eq.~15.4.6]{HMF} to compute the hypergeometric functions yields the lemma.
\end{proof}

\begin{lem}\label{lem:inth} Let $\omega(2)=0$. Then
\begin{equation*}
\int_{-\infty}^{\infty}h(\omega;1;r)dr=0.
\end{equation*}
\end{lem}
\begin{proof}
Using the inversion formula for the cosine Fourier transform \cite[Eqs. 1.14.9, 1.14.11]{HMF}, we obtain
\begin{equation*}
\int_{-\infty}^{\infty}h(\omega;1;r)dr=\frac{\pi^{3/2}}{2}\omega(2)=0.
\end{equation*}
\end{proof}

\section{Analytic continuation}
\subsection{Convolution formula at the point $s=1$}

The analytic continuation of the convolution formula to the region containing the special point $s=1$ is quite delicate because of the presence of $\zeta(s)$ in the continuous spectrum \eqref{eq:zc}. In particular, it is required to show that the integral in \eqref{eq:zc} vanishes at $s=1$. With this goal, we let $l=1$ and use Corollary \ref{cor:h11}. As a result, we establish the explicit formula given by Theorem \ref{thm:at1}, relating norms of prime geodesics to moments of symmetric-square $L$-functions.
In order to prove Theorem \ref{thm:at1}, it is required to continue analytically $Z_C(s)$ given by \eqref{eq:zc}.

\begin{lem}\label{lem:contcspec}
Let $\omega(x)=0$ for $x\leq 2.$ We have
\begin{multline*}
\lim_{s \rightarrow 1}\frac{\zeta(s)}{\pi^{1/2+s}}\int_{-\infty}^{\infty}
\frac{\left|\zeta(s+2ir)\right|^2}{\left|\zeta(1+2ir\right)|^2}h(\omega;s;r)dr=\\
-\frac{1}{2\sqrt{\pi}}h(\omega;1;0)
+\frac{1}{\pi^{3/2}}\int_{-\infty}^{+\infty}\frac{\partial}{\partial s}h\left( \omega;s;r\right)\biggr|_{1}dr
\\+\frac{1}{\pi^{3/2}}\int_{-\infty}^{+\infty}\biggl(
\frac{\zeta'(1+2ir)}{\zeta(1+2ir)} +\frac{\zeta'(1-2ir)}{\zeta(1-2ir)}\biggr)h(\omega;1;r)dr.
\end{multline*}
\end{lem}

\begin{proof}
 Consider
\begin{equation*}
I(s):=\frac{\zeta(s)}{\pi^{1/2+s}}\int_{-\infty}^{\infty}
\frac{\zeta(s+2ir)\zeta(s-2ir)}{\zeta(1+2ir) \zeta(1-2ir)}h(\omega;s;r)dr.
\end{equation*}
Our goal is to show that the pole of $\zeta(s)$ at $s=1$ is compensated by the vanishing integral (see Lemma \ref{lem:inth}).

Making the change of variables $z:=2ir$, we have
\begin{equation*}
I(s)=\zeta(s)\pi^{1/2-s}\frac{1}{2\pi i}\int_{(0)}
\frac{\zeta(s+z)\zeta(s-z)}{\zeta(1+z) \zeta(1-z)}h\left(\omega;s;\frac{z}{2i}\right)dz.
\end{equation*}
The integrand has poles at $z=\pm( s-1)$ and also at the zeros of $\zeta(1\pm z)$.
By Cauchy's theorem
\begin{multline*}
I(s)=\zeta(s)\pi^{1/2-s}\frac{1}{2\pi i}\int_{C_{\delta}}
\frac{\zeta(s+z)\zeta(s-z)}{\zeta(1+z) \zeta(1-z)}h\left(\omega;s;\frac{z}{2i}\right)dz+\\
\pi^{1/2-s}\frac{\zeta(2s-1)}{\zeta(2-s)}h\left(\omega;s;\frac{1-s}{2i}\right),
\end{multline*}
where the contour of integration is defined as
\begin{equation*}C_{\delta}=\gamma_1\cup \gamma_2 \cup \gamma_3,
\end{equation*}
\begin{equation*}
\gamma_1=(-i\infty,-i\Im{s}-i\delta), \quad \gamma_3=(-i\Im{s}+i\delta,i\infty),
\end{equation*}
and $\gamma_2$ is a semicircle of radius $\delta>0$ with a center at $-i\Im{s}$  for $\Re{z}<0$.
 The parameter $\delta$ is chosen such that all zeros of $\zeta(1+z)$ are located to the left of $C_{\delta}$.

The next step is to compute the limit when $s$ tends to $1$. Evaluating the limit for the second summand, we find
\begin{equation*}
\lim_{s\rightarrow 1}\pi^{1/2-s}\frac{\zeta(2s-1)}{\zeta(2-s)}h\left(\omega;s;\frac{1-s}{2i}\right)=-\frac{1}{2\sqrt{\pi}}h(\omega;1;0).
\end{equation*}
For simplicity, let
\begin{equation*}
F(s):=\pi^{1/2-s}\frac{1}{2\pi i}\int_{C_{\delta}}
\frac{\zeta(s+z)\zeta(s-z)}{\zeta(1+z) \zeta(1-z)}h\left(\omega;s;\frac{z}{2i}\right)dz.
\end{equation*}
Then the limit of the first summand can be found as follows
\begin{multline*}
\lim_{s\rightarrow 1}\zeta(s)F(s)=F'(1)=\frac{1}{2\pi i}\frac{1}{\sqrt{\pi}}\int_{C_{\delta}}\frac{\partial}{\partial s}h\left( \omega;s;\frac{z}{2i}\right)\biggr|_{1}dz\\+
\frac{1}{2\pi i}\frac{1}{\sqrt{\pi}}\int_{C_{\delta}}\left( \frac{\zeta'(1+z)}{\zeta(1+z)}+\frac{\zeta'(1-z)}{\zeta(1-z)}-\log{\pi}\right)h\left( \omega;1;\frac{z}{2i}\right)
dz.
\end{multline*}
Note that the expression
\begin{equation*}
\frac{\zeta'(1+z)}{\zeta(1+z)}+\frac{\zeta'(1-z)}{\zeta(1-z)}
\end{equation*}
is analytic at $z=0$.
Now moving the contour back to $\Re{z}=0$ without crossing any poles and making the change of variables $r:=z/(2i)$, we obtain
\begin{multline*}
\lim_{s\rightarrow 1}\zeta(s)F(s)=\frac{1}{\pi^{3/2}}\int_{-\infty}^{+\infty}\biggl(
\frac{\zeta'(1+2ir)}{\zeta(1+2ir)} +\frac{\zeta'(1-2ir)}{\zeta(1-2ir)}\biggr)h(\omega;1;r)dr\\+
\frac{1}{\pi^{3/2}}\int_{-\infty}^{+\infty}\frac{\partial}{\partial s}h\left( \omega;s;r\right)\biggr|_{1}dr
.
\end{multline*}
Note that the summand with $\log{\pi}$ disappeared because it is identically zero by Lemma \ref{lem:inth}. The convergence of the first integral follows from \eqref{eq:hatone} using the integration by parts.
The last thing to verify is that the integral
\begin{equation*}
\int_{-\infty}^{+\infty}\frac{\partial}{\partial s}h\left( \omega;s;r\right)\biggr|_{1}dr
\end{equation*}
converges. Note that  $h_2(\omega;s;r)=0$ since $\omega(x)=0$ for $x\leq 2$ and $l=1$. Differentiating \eqref{eq:hh1h2}, we obtain
\begin{equation*}
\frac{\partial}{\partial s}h\left( \omega;s;r\right)\biggr|_{1}=h_3(\omega;r)+h_3(\omega;-r)+h_4(\omega;r),
\end{equation*}
where
\begin{multline*}
h_3(\omega;r):=\int_{2l}^{\infty}\frac{\omega(x)}{2x} \left( \frac{x}{2}\right)^{-2ir}\frac{\partial}{\partial s}\Biggl(\Gamma(1/2+s/2+ir)\\ \times \frac{\Gamma(s/2+ir)}{\Gamma(1+2ir)}
F(s/2+ir,s/2+1/2+ir,1+2ir;4l^2/x^2)\Biggr)\biggr|_{1}dx
\end{multline*}
and
\begin{multline*}
h_4(\omega;r):=-\pi^{1/2}\int_{0}^{\infty}\frac{\omega(x)\log{x}}{\sqrt{x^2-4}}\cos{\left(2r\log{\left(\frac{x}{2}+
\frac{\sqrt{x^2-4}}{2} \right)}\right)}dx\\
-\frac{\pi^{3/2}}{2}\frac{\cosh{\pi r}}{\sinh{\pi r}}\int_{0}^{\infty}\frac{\omega(x)}{\sqrt{x^2-4}}\sin{\left(2r\log{\left(\frac{x}{2}
\frac{\sqrt{x^2-4}}{2} \right)}\right)}dx.
\end{multline*}
Integrating by parts, one shows that all integrals in $h_4(\omega;r)$ are of rapid decay.
Now consider $h_3(\omega;r)$. Using \cite[Eq.~15.6.1]{HMF}, we obtain
\begin{multline*}
\Gamma(1/2+s/2+ir)\frac{\Gamma(s/2+ir)}{\Gamma(1+2ir)}
F(s/2+ir,s/2+1/2+ir,1+2ir;4l^2/x^2)\\= \frac{\Gamma(1/2+s/2+ir)}{\Gamma(1-s/2+ir)}\int_{0}^{1}
\frac{y^{s/2-1+ir}(1-y)^{-s/2+ir}}{(1-4y/x^2)^{1/2+s/2+ir}}.
\end{multline*}
Differentiating the above with respect to $s$ at the point $1$, changing the order of integration over $x$ and $y$, computing the integral over $x$ by parts, we prove that $h_3(\omega;r)$ is also of rapid decay. This completes the proof of Lemma \ref{lem:contcspec}.

\end{proof}

Now Theorem \ref{thm:at1} follows from Theorem \ref{thm:convform} and Lemma \ref{lem:contcspec}. For this, it is required to show the convergence of all series and integrals in
$Z_H(1)$ and $Z_D(1)$. In the case of $Z_H(1)$, this  is a consequence of \eqref{eq:gatone}.
In the case of $Z_D(1)$, the convergence follows from \eqref{eq:hatone} using the integration by parts formula.

\subsection{Convolution formula at the point $s=1/2$}
Next, we derive a convolution formula at the special point $s=1/2$. With this goal, we first continue the formula to the region $\Re{s}<1$, $s\neq 1/2$.
\begin{thm} Assume that $\Re{s}<1$, $s\neq 1/2$ and $\omega(2l)=0$. Then
\begin{multline}\label{eq:cont2}
\sum_{n=1}^{\infty}\omega(n)\mathscr{L}_{n^2-4l^2}(s)=\widetilde{\omega}(1)\sigma_{-s}(l^2)\frac{\zeta(2s)}{\zeta(1+s)}\\+\frac{2^s\cos{(\pi s/2)}}{\pi^{s-1/2}}\sigma_{s-1}(l^2)\frac{\zeta(2s-1)}{\zeta(2-s)}\Gamma(s-1/2)\int_{0}^{2l}\omega(x)(4l^2-x^2)^{1/2-s}dx\\+\frac{2^s\sin{(\pi s/2)}}{\pi^{s-1/2}}\sigma_{s-1} (l^2)\frac{\zeta(2s-1)}{\zeta(2-s)}\Gamma(s-1/2)\int_{2l}^{\infty}\omega(x)(x^2-4l^2)^{1/2-s}dx\\
+Z_C(s)+Z_H(s)+ Z_D(s),
\end{multline}
where  $Z_C(s)$, $Z_H(s)$, $Z_D(s)$ are defined by \eqref{eq:zc}, \eqref{eq:zh} and \eqref{eq:zd}.
\end{thm}
\begin{proof}

In order to obtain the analytic continuation to the region $\Re{s}<1$, we only need to consider $Z_C(s)$.
Therefore, it is required to prove the analytic continuation of the following integral (see \eqref{eq:zc})
\begin{equation*}
\zeta(s)\pi^{1/2-s}\frac{1}{2\pi i}\int_{(0)}\frac{\sigma_{z}(l^2)}{l^{z}}
\frac{\zeta(s+z)\zeta(s-z)}{\zeta(1+z) \zeta(1-z)}h\left(\omega;s;\frac{z}{2i}\right)dz.
\end{equation*}
The integrand has poles at $z=\pm( s-1)$ and also at the zeros of $\zeta(1\pm z)$.

Thus taking sufficiently small $\delta>0$ we can change the contour of integration  without crossing any poles to
\begin{equation*}C_{\delta}=\gamma_1\cup \gamma_2 \cup \gamma_3\cup \gamma_4\cup \gamma_5,
\end{equation*}
\begin{equation*}
\gamma_1=(-i\infty,-i\Im{s}-i\delta), \quad \gamma_3=(-i\Im{s}+i\delta,i\Im{s}-i\delta),\quad
\gamma_5=(i\Im{s}+i\delta,i\infty),
\end{equation*}
and $\gamma_2$ is a semicircle of radius $\delta>0$ with a center at $-i\Im{s}$  for $\Re{z}>0$,
$\gamma_4$ is a semicircle of radius $\delta>0$ with a center at $i\Im{s}$  for $\Re{z}<0$.
This yields the analytic continuation of $Z_C(s)$ to the region $\Re{s}>1-\delta$.
Now changing the contour of integration back to the line $\Re{z}=0$ and applying
Cauchy's theorem, we get
\begin{multline*}
\zeta(s)\pi^{1/2-s}\frac{1}{2\pi i}\int_{(0)}\frac{\sigma_{z}(l^2)}{l^{z}}
\frac{\zeta(s+z)\zeta(s-z)}{\zeta(1+z) \zeta(1-z)}h\left(\omega;s;\frac{z}{2i}\right)dz+\\
2\pi^{1/2-s}\frac{\sigma_{s-1}(l^2)}{l^{s-1}}\frac{\zeta(2s-1)}{\zeta(2-s)}h\left(\omega;s;\frac{s-1}{2i}\right).
\end{multline*}
Note that we used the fact that $h(\omega;s;r)$ is the even function in $r$, which is  a consequence of Lemma \ref{lem:hcont}.
This yields the analytic continuation to $\Re{s}<1$. Using Corollary \ref{cor:h2i}, the assertion follows.

\end{proof}

Now we extend the region to $s=1/2$ by letting $s=1/2+u$ and computing the limit as $u$ tends to zero.

\begin{thm} \label{thm:at12} Assume that $\omega(2l)=0$. Then
\begin{multline*}
\sum_{n=1}^{\infty}\omega(n)\mathscr{L}_{n^2-4l^2}(1/2)=\frac{\sigma_{-1/2}(l^2)}{2\zeta(3/2)}
\int_{0}^{\infty}\omega(x)\Biggl(
\log{|x^2-4l^2|}-\\-\frac{\pi}{2}\sgn(x-2l)+3\gamma-2\frac{\zeta'(3/2)}{\zeta(3/2)}-\log{8\pi}-
\frac{2}{\sigma_{-1/2}(l^2)}\sum_{d|l^2}d^{-1/2}\log{d}
\Biggr)\\
+Z_C(1/2)+Z_H(1/2)+ Z_D(1/2),
\end{multline*}
where  $Z_C(s)$, $Z_H(s)$, $Z_D(s)$ are defined by \eqref{eq:zc}, \eqref{eq:zh} and \eqref{eq:zd}.
\end{thm}
\begin{proof}
 Consider equation \eqref{eq:cont2} with $s=1/2+u$.
Using the functional equation for the Riemann zeta function
\begin{equation*}
\zeta(2u)=2^{2u}\pi^{2u-1}\sin{\pi u}\Gamma(1-2u)\zeta(1-2u)
\end{equation*}
and the functional equation for $\Gamma(u)$,
the main terms in equation \eqref{eq:cont2} give the  following
\begin{multline*}
\frac{\zeta(1+2u)}{\zeta(3/2+u)}\sigma_{-1/2-u}(l^2)\int_{0}^{\infty}\omega(x)+
\frac{\zeta(1-2u)}{\zeta(3/2-u)}\sigma_{-1/2+u}(l^2)2^{1/2+3u}\pi^u\frac{\Gamma(1-2u)}{\Gamma(1-u)} \\ \biggl( \sin{\left( \frac{\pi}{4}+\frac{\pi u}{2}\right)}\int_{2l}^{\infty}\omega(x)(x^2-4l^2)^{-u}+ \cos{\left( \frac{\pi}{4}+\frac{\pi u}{2}\right)}\int_{0}^{2l}\omega(x)(4l^2-x^2)^{-u}\biggr).
\end{multline*}
Note that the expression above is holomorphic at $u=0$. Letting $u\rightarrow 0$ and using the L'H\^{o}pital rule, we prove the  theorem.

\end{proof}

\section{Applications}
This section is devoted to proving Theorems \ref{lem:convsum} and \ref{lem:exp}. In order to do this we need some preliminary results.

First, we recall the subconvexity estimate for the generalized Dirichlet $L$-functions. For $n\neq 0$ and any $\epsilon>0$ we have
\begin{equation}\label{eq:subL}
\mathscr{L}_n(1/2)\ll n^{\theta+\epsilon},
\end{equation}
where $\theta$ is a subconvexity exponent for the Dirichlet $L$-functions of real primitive characters. See \cite[Lemma~4.2]{BF1} for details.
It follows  from the result of Conrey and Iwaniec \cite{CI} that $\theta =1/6$.

Second, we need the following lemma.
\begin{lem}\label{estint}
Let $f(x), g(x)\in C^{1}[a,b]$ are real functions and $g(x)/f'(x)$ is a piecewise monotone function. Then
\begin{equation*}
\int_{a}^{b}g(x)\exp(if(x))dx\ll \max_{a<x<b}\left(\frac{g(x)}{f'(x)} \right).
\end{equation*}
\end{lem}
\begin{proof}
See \cite[Chapter~4.1]{T}.
\end{proof}

Third, we prove the uniform approximation formula for the special case of the Gauss hypergeometric function.
\begin{lem}Let
\begin{equation*}
U(\xi)=\frac{\tanh^{1/2}(\xi/2)}{\cosh^{2ir}(\xi/2)}F\left( \frac{1}{4}+ir,\frac{3}{4}+ir,1+2ir; \frac{1}{\cosh^2{\xi/2}}\right).
\end{equation*}
There is a $\xi_0$ such that for $r \rightarrow \infty$ uniformly for  $\xi>\xi_0$ we have
\begin{equation*}
U(\xi)=2^{2ir}\exp(-ir\xi)\left(1+\frac{1-\coth{\xi}}{16ir} \right)+O\left( \frac{1}{\exp(\xi)r^2}\right).
\end{equation*}
\end{lem}
\begin{proof}
Using the differential equation for the hypergeometric function, we find that
\begin{equation*}
U''(\xi)+\left( r^2+\frac{1}{16\sinh^2{\xi/2}}\right)U(\xi)=0.
\end{equation*}
Now the required asymptotic formula can be  obtained by following the proof of \cite[Lemma~2.4]{Zav}.
\end{proof}

\begin{cor}\label{cor:hyperest} There is an $x_0>2$ such that for $r \rightarrow \infty$ uniformly for  all $x>x_0$ we have
\begin{multline*}
F\left( \frac{1}{4}+ir,\frac{3}{4}+ir,1+2ir; \frac{4}{x^2}\right)=x^{2ir}\exp(-2ir\arcosh{x/2})\\ \times \left( \frac{x^2}{x^2-4}\right)^{1/4}
\left(1+\frac{1}{16ir }\left( 1-\frac{x^2-2}{x\sqrt{x^2-4}}\right) \right)+O\left(\frac{1}{x^2r^2} \right).
\end{multline*}
\end{cor}

Next, we investigate the moments of generalized Dirichlet $L$-functions in long and short intervals.
\begin{thm}\label{thm:appl1}
Let $\omega(x)$ be an infinitely differentiable function such that $\omega(x)=1$ for $T<x<X-T$, $\omega(x)=0$ for $x>X$ and $x<T/2$. For any $\epsilon>0$
\begin{multline*}
\sum_{2<n<X}\omega(n)\mathscr{L}_{n^2-4}(1/2)=\frac{1}{2\zeta(3/2)}
\int_{0}^{\infty}\omega(x)\Biggl(
\log{|x^2-4|}-\\-\frac{\pi}{2}\sgn(x-2)+3\gamma-2\frac{\zeta'(3/2)}{\zeta(3/2)}-\log{8\pi}
\Biggr)dx+O\left( \sqrt{X}\left(\frac{X}{T}\right)^{1/2+\epsilon}\right).
\end{multline*}

\end{thm}

\begin{proof}

Our goal is to estimate $Z_C(1/2)$, $ Z_H(1/2)$ and $Z_D(1/2)$.
On the one hand, using Corollary \ref{cor:hyperest} and Lemma \ref{estint}, we estimate the function $h_1(\omega;1/2;\pm r)$ given by equation \eqref{eq:hwsr}. Consequently,
\begin{equation*}
h(\omega;1/2;r)\ll  \frac{\sqrt{X}}{r^{3/2}}.
\end{equation*}
On the other hand, application of Corollary \ref{cor:hyperest} followed by integration  by parts and Lemma \ref{estint} yields
\begin{equation*}
h(\omega;1/2;r)\ll  \frac{X^{3/2}}{Tr^{5/2}}.
\end{equation*}
Using the last estimates, equation \eqref{eq:estsym2} and dyadic partition of unity over $t_j$, we prove that
\begin{equation*}
Z_D(1/2)\ll \sqrt{X}\left(\frac{X}{T}\right)^{1/2+\epsilon}.
\end{equation*}
Standard estimates on the Riemann zeta function yield that the same estimate is satisfied for $Z_C(1/2)$.
Using Lemma \ref{lem:g}, we obtain
\begin{equation*}
g(\omega;1/2;k)\ll \int_{0}^{\infty}\omega(x)x^{1/2}\Psi_k\left( \frac{4}{x^2}\right)dx,
\end{equation*}
where $\Psi_k(x)$ is defined by \eqref{eq: psisym2}. Then \cite[Theorem~6.17]{BF1} yields
\begin{equation*}
g(\omega;1/2;k)\ll \exp(-ck)
\end{equation*}
for some $c>0$. Therefore, $Z_H(1/2)\ll 1.$
\end{proof}
\subsection{Proof of Theorem \ref{lem:convsum}}

To remove the dependence on $\omega(n)$ in Theorem \ref{thm:appl1}, we use the subconvexity estimate \eqref{eq:subL}, obtaining the additional error term $O\left( X^{2\theta}T\right)$. Choosing $T=X^{2/3-4\theta/3}$,  we prove the asymptotic expansion \eqref{eq:app1}.
\subsection{Proof of Theorem \ref{lem:exp}}

The proof is similar to Theorem \ref{thm:appl1}. First, we estimate equation \eqref{eq:hwsr} by absolute value, obtaining
\begin{equation*}
h(\omega;1/2;r)\ll \frac{T}{\sqrt{rX}}.
\end{equation*}
Second, using Corollary \ref{cor:hyperest} and integrating by parts two times, we have
\begin{equation*}
h(\omega;1/2;r)\ll  \frac{X^{3/2}}{Tr^{5/2}}.
\end{equation*}
Consequently, we prove that
\begin{equation*}
\sum_{n>2}\mathscr{L}_{n^2-4}(1/2)\exp\left( -\left(\frac{n-X}{T} \right)^2\right)\ll
T+\sqrt{X}\left(\frac{X}{T}\right)^{1/2+\epsilon}.
\end{equation*}

\section*{Acknowledgments}
We thank the anonymous referees for many valuable comments and suggestions on the manuscript.

\nocite{}

\end{document}